\documentclass[11pt]{article}

\usepackage[utf8]{inputenc}
\usepackage[T1]{fontenc}
\usepackage{lmodern}
\usepackage[margin=1in]{geometry}
\usepackage{amsmath,amssymb,amsthm}
\usepackage{graphicx}
\usepackage{booktabs}
\usepackage{hyperref}
\usepackage{xcolor}
\usepackage{authblk}         % nicer author/affiliation blocks

\hypersetup{
	colorlinks=true,
	linkcolor=blue,
	citecolor=blue,
	urlcolor=blue,
}

\newtheorem{theorem}{Theorem}
\newtheorem{lemma}[theorem]{Lemma}

\newtheorem{corollary}[theorem]{Corollary}
\theoremstyle{definition}
\newtheorem{definition}{Definition}
\newtheorem{example}{Example}
\theoremstyle{remark}
\newtheorem{remark}{Remark}

\usepackage{graphicx}%
\usepackage{multirow}%
\usepackage{amsmath,amssymb,amsfonts}%
\usepackage{amsthm}%
\usepackage{mathrsfs}%
\usepackage[title]{appendix}%
\usepackage{xcolor}%
\usepackage{textcomp}%
\usepackage{manyfoot}%
\usepackage{booktabs}%
\usepackage{algorithm}%
\usepackage{algorithmicx}%
\usepackage{algpseudocode}%
\usepackage{listings}%

\def\disp{\displaystyle}

\def\tto{\;{\lower 1pt \hbox{$\rightarrow$}}\kern -10pt
	\hbox{\raise 2pt \hbox{$\rightarrow$}}\;}

\def\epsilon{\varepsilon}

\def\R{\Bbb R}

\def\N{\Bbb N}

\title{Dynamical Systems and Nonconvex Quadratic Programs}

\author[1]{Nguyen Nang Thieu}
\affil[1]{Department of Optimization and Scientific Computing, Institute of Mathematics, Vietnam Academy of Science and Technology, 18 Hoang Quoc Viet, Hanoi, 10307, Vietnam\\ \texttt{nnthieu@math.ac.vn}}

\date{\empty}
\begin{document}
\maketitle

\begin{abstract}
This paper studies a dynamical system approach for solving indefinite quadratic programming problems subject to linear constraints. We investigate the convergence of the trajectory generated by the system to a Karush-Kuhn-Tucker point of the quadratic programs. In addition, we derive an estimate for the distance between the trajectory and a solution of the problem. We further prove that the objective value is decreasing along the trajectory. An illustrative example and a numerical test are presented to demonstrate the behavior and performance of the proposed method.
\end{abstract}


\maketitle

\section{Introduction}\label{Sect-1}

In this paper, we consider the \textit{indefinite quadratic programming problem under linear constrains} (IQPs) given as follows
\begin{eqnarray}\label{QP problem}
	\min\Big\{f(x)=\dfrac{1}{2}x^\top Qx+q^\top x \mid Ax\geq b\Big\}, \end{eqnarray}
where $Q\in\mathbb R^{n\times n}$ and $A\in\mathbb  R^{m\times n}$ are given matrices, $Q$  is symmetric, $q\in\mathbb  R^n$ and $b\in\mathbb  R^m$ are given vectors, and the superscript $^\top $ signifies the matrix transposition. Throughout the paper, we assume that the constraint set $C:=\big\{x\in\mathbb R^n\mid Ax\geq b\big\}$ of~\eqref{QP problem} is nonempty. 

This class of problems naturally emerge in nonlinear programming frameworks, especially within sequential quadratic programming techniques, where each iteration requires the solution of a quadratic subproblem. Classical schemes of this type include the methods developed by Wilson, Pang, and Maratos-Mayne-Polak, together with several globalized variants; see~\cite[Section~2.9]{Polak_1997}. A broad overview of theoretical developments in indefinite quadratic optimization is presented in~\cite{Bomze_1998}. 

Indefinite quadratic models appear in numerous real-world applications. In operations research and economics, they have been used to model production planning, agricultural systems, and policy-related decision processes~\cite{Gupta_1995}. In recent years, the scope of applications has expanded considerably due to developments in data science and engineering. In particular, quadratic optimization formulations arise in portfolio selection, image processing, support vector machine training, and several machine learning tasks~\cite{Akoa_2008,Cornuejols_2018,Liu et al_2017a,McCarl et al_1977,Xu et al_2017,Xue et al_2019}.

The mathematical properties of indefinite quadratic programming have been extensively studied in the literature. In particular, these studies address topics such as the existence of optimal solutions, characterization of stationary points, optimality conditions, sensitivity analysis, and stability of solution mappings. Detailed discussions of these aspects can be found in~\cite{lty2012}. Since indefinite quadratic programming is generally NP-hard~\cite{Pardalos_Vavasis_1991}, the development of efficient algorithms remains a challenging topic, particularly for large-scale data. Most currently available methods are designed to compute stationary solutions or local minimizers rather than globally optimal points. Some well known numerical approaches include branch-and-bound techniques, decomposition methods, and local descent procedures; see~\cite{Bomze_Danninger_1994,Cambini_Sodini_2005,PhamDinh_LeThi_2,PhamDinh_LeThi98,PhamDinh_LeThi_3,PhamDinh_LeThi_Akoa,Ye89,Ye92,Ye97}. 

\medskip
Besides discrete methods, some continuous methods are proposed to solve optimization problems. Antipin~\cite{Antipin_94} introduced a projected dynamical system for minimizing a continuously differentiable function over a closed convex set by combining the gradient of the objective function with the metric projection onto the feasible set. Convergence properties for both first-order and second-order systems were established therein. Continuous dynamical systems have also been extensively applied to variational inequalities. In particular, Cavazzuti, Pappalardo, and Passacantando~\cite{CPP_02} considered a projected dynamical system associated with a strongly monotone and Lipschitz continuous operator. Additional results on dynamical approaches to optimization problems, variational inequalities, and equilibrium problems can be found in~\cite{AGR_00,APR_14,BCPP_19,DN_93,HSV_18,Hai_2022,NZ_96,PP_02,VTV_22,Vuong_2021,VS_20}. 

In~\cite{PhamDinh_LeThi_Akoa}, Pham Dinh and coauthors developed several decomposition-based algorithms for solving indefinite quadratic programming problems, including projection-type and proximal-type DC schemes. These methods rely on DC programming and DCA (Difference-of-Convex functions Algorithm) techniques introduced by Pham Dinh Tao and Le Thi Hoai An~\cite{PhamDinh_LeThi_AMV97,PhamDinh_LeThi98}; see also~\cite{LeThi_PhamDinh_AOR05,PhamDinh_LeThi_4}. Their theoretical properties and numerical performance were subsequently investigated in terms of convergence behavior, robustness, and convergence speed in~\cite{ATY2,CLY_2024,Tuan_JMAA2015,Tuan_JOTA2015}.

To the best of our knowledge, Pappalardo, Thieu, and Yen~\cite{pty_2025} were the first to propose a DC-based dynamical approach for solving indefinite quadratic programming problems, using the DC decomposition introduced in~\cite{CLY_2024}. They showed that the strong pseudomonotonicity assumption adopted in several previous works is rather restrictive, even for affine operators defined on compact polyhedral convex sets. Moreover, they established convergence of the trajectories of the dynamical system to Karush-Kuhn-Tucker points in the one-dimensional case on the interval $[-1,1]$. In~\cite{pty_2025}, the authors proposed several open question related to the convergence of trajectories to a KKT point in more general cases for indefinite quadratic programming problems. Very recently, Niu~\cite{N_2026} proposed a continuous-time scheme for damped DCA and proved convergence to a critical point for bounded trajectories under a Kurdyka-\L{}ojasiewicz hypothesis.

In this paper, we consider one of the dynamical system schemes proposed in~\cite{pty_2025} and apply it to indefinite quadratic programming problems under linear constraints, where the feasible set may be unbounded. One difference between our approach and the method proposed in~\cite{pty_2025}, as well as other DC decomposition-based methods~\cite{PhamDinh_LeThi_AMV97,PhamDinh_LeThi98} for IQPs, is that our approach does not require the decomposition parameter $\rho$ to be strictly greater than the largest eigenvalue of $Q$. Consequently, there is no need to compute the largest eigenvalue of $Q$ or an upper bound on its eigenvalues. We establish the asymptotic convergence of the trajectory
generated by the dynamical system to the Karush-Kuhn-Tucker set of~\eqref{QP problem}. This partially answers Question~2 in~\cite{pty_2025} for the case where the constraints are affine. In addition, we obtain an estimate for the distance between the trajectory and a solution of~\eqref{QP problem}. The value of the objective function is decreasing along the trajectory generated by the dynamical system. In some sense, this property inherits from the known feature of DCA sequences: the value of the objective function decreases along any DCA sequence (see~\cite{PhamDinh_LeThi_AMV97,PhamDinh_LeThi98}).

The paper is organized as follows. In Section~\ref{Sect-2}, we provide some preliminaries on IQPs and a continuous scheme for solving them. Several results concerning this scheme, including a convergence theorem and properties of the trajectory, are established in Section~\ref{Sect-3}. Section~\ref{Sect-4} presents two examples to illustrate the convergence of the generated trajectory and compare the numerical performance of the continuous scheme with that of the DCA method. Finally, concluding remarks and directions for further research are given in Section~\ref{Sect-5}.
\section{Preliminaries}\label{Sect-2}

In this section, we recall some notions and results that will be used throughout the paper. The notation $[a,b]$ (resp., $(a,b)$) stands for a closed interval (resp., an open interval) in the real line $\mathbb{R}$. The set of positive integers and the set of nonnegative real numbers are denoted respectively by $\mathbb{N}$  and $\mathbb{R}_+$. Let $\R^n$ be the $n$-dimensional Euclidean space equipped with the norm $ \Vert \cdot \Vert $ and the scalar product~$\langle \cdot,\cdot \rangle$. The \textit{metric projection} of a point $x\in \R^n$ onto a set $C$ is defined by $$P_C(x) = \big\{y\in C \mid \|y-x\| \inf_{z\in C}\|z-x\|\big\}.$$ 

\begin{definition} {\rm A point $x\in\mathbb R^n$ is said to be a {\it Karush-Kuhn-Tucker point} (a {\it KKT point}) of~\eqref{QP problem} if there exists a multiplier $\lambda\in\mathbb R^m$  such that
		\begin{eqnarray}\label{KKT_Point_Set}
			\begin{cases}Qx+q-A^T\lambda=0,\\ Ax\geq b,\ \; \lambda\geq 0,\ \; \lambda^T(Ax-b)=0.\end{cases}
		\end{eqnarray} 
		The set of all KKT points of~\eqref{QP problem} is denoted by $C_*$.
} \end{definition}

If $\bar x$ is a local solution of~\eqref{QP problem}, then $\bar x$ is a KKT point of the problem, that is, $\bar x\in C_*$ (see~\cite[Theorem~3.3]{lty05}). It is well known that a point $x\in C_*$ if and only if $x$ is a solution of the following \textit{affine variational inequality}
\begin{eqnarray}\label{AVI} x\in C,\quad \langle Qx+q,u-x\rangle\geq 0\ \; \forall u\in C;
	\end{eqnarray}
	see, e.g.,~\cite[Theorem~5.3]{lty05} and~\cite{lty2012}.

The next important lemma describes the special structure of the KKT set of~\eqref{QP problem}.

		\begin{lemma}\label{lem:KKT_components} {\rm (\cite[Lemma~3.1 and its proof]{LT_1992}; see also \cite[Lemma 2.2]{Tuan_JMAA2015})}\ \, The KKT set of~\eqref{QP problem} has finitely many connected components.	Let $C_1, C_2, \cdots, C_r$ denote the connected components of $C_*$. Then, the following properties are valid.
			\begin{itemize}
				\item[{\rm (a)}] Each component $C_i$ is the union of finitely many polyhedral convex sets.
				\item[{\rm (b)}] The closed sets $C_i$, $i=1,\ldots r$, are properly separated each from others, that is, there exists $\delta>0$ such that if  $i\neq j$ then
				\begin{equation}\label{positive_excess}
					\inf\limits_{x\in C_i }d(x,C_j)\geq \delta.
				\end{equation}
				\item[{\rm (c)}] The objective function of~\eqref{QP problem} is constant on each component $C_i$.
			\end{itemize}
	\end{lemma}

We will need the following local error bound, which is also valid for the solution sets of affine variational inequalities.

 \begin{lemma}\label{lem:dist_kkt_set}{\rm (See~\cite[Lemma 2.1]{Tuan_JMAA2015}, cf.~\cite[Lemma~3.1]{LT_1992})} \textit{For any $\rho>0$, if $C_*\neq \emptyset$, then there exist scalars $\varepsilon>0$ and $\ell>0$ such that
 	\begin{equation}\label{ERB}
 		d(x,C_*)\leq \ell \left\|P_C\Big(x-\frac{1}{\rho}(Qx+q)\Big)-x\right\|
 	\end{equation}
 	for all $x\in C$ with
 	\begin{equation}\label{locality}
 		\left\|P_C\Big(x-\frac{1}{\rho}(Qx+q)\Big)-x\right\|<\varepsilon.
 	\end{equation}}
 \end{lemma}
 
To solve~\eqref{QP problem} by DCA, one uses the decomposition  $f(x)=f_1(x)-f_2(x)$ for the objective function $f$ with
\begin{eqnarray*}
	f_1(x)
	=
	\Big[\frac{1}{2}x^\top Q_1x+q^\top x\Big]
	+\delta_C(x)
	\quad {\rm and} \quad
	f_2(x)
	=
	\frac{1}{2}x^\top Q_2x,
\end{eqnarray*}
where $Q_1,Q_2\in \mathbb R^{n\times n}$ are symmetric positive definite matrices satisfying $Q=Q_1-Q_2$. Here, $\delta_C$ denotes the \textit{indicator function} of $C$, that is,
\[
\delta_C(x)=
\begin{cases}
	0 & \text{if } x\in C,\\
	\infty & \text{if } x\notin C.
\end{cases}
\]

Let $\lambda_{\rm min}(Q)$ and $\lambda_{\rm max}(Q)$ be, respectively, the smallest eigenvalue and the largest eigenvalue of $Q$. Following Pham Dinh et al.~\cite{PhamDinh_LeThi98,PhamDinh_LeThi_Akoa}, one may choose the decomposition matrices in one of the following ways:
\begin{itemize}
	\item[(a)]
	$Q_1=\rho I$ and $Q_2=\rho I-Q$, where $\rho >0$ is a real number such that $\rho>\lambda_{\rm max}(Q)$;
	
	\item[(b)]
	$Q_1=Q+\rho I$ and $Q_2=\rho I$, where $\rho >0$ is a real number satisfying $\rho>-\lambda_{\rm min}(Q)$.
\end{itemize}

Since upper and lower bounds for the eigenvalues of a symmetric matrix can be computed efficiently (see~\cite[estimate~(a) or estimate~(b), p.~418]{Stoer_Bulirsch_1980}), one can readily determine a parameter $\rho>0$ satisfying the above conditions.

\medskip

In this paper, we focus on the decomposition given in~(a). This leads to the explicit iteration scheme
\begin{equation}\label{Iteration_A}
	x^{k+1}
	:=
	P_C\Bigl(
	x^k-\frac{1}{\rho}(Qx^k+q)
	\Bigr),
	\qquad
	k=0,1,2,\dots,
\end{equation}
where $x^0\in C$ is an initial point (see, e.g.,~\cite{PhamDinh_LeThi_Akoa}). The continuous dynamical system associated with the iteration scheme~\eqref{Iteration_A}, which was introduced in~\cite{pty_2025}, is as follows
\begin{equation}\label{dynamic_sys_A}
	\begin{cases}
		\dot{x}(t)
		=
		\dfrac{1}{\eta}
		\left[
		P_C\left(
		x(t)-\dfrac{1}{\rho}\big(Qx(t)+q\big)
		\right)
		-x(t)
		\right],
		\quad t\geq 0,
		\\[2ex]
		x(0)=x^0,
	\end{cases}
\end{equation}
where $\eta$ and $\rho$ are fixed positive constants\footnote{The original version of the dynamical system~\eqref{dynamic_sys_A} in~\cite{pty_2025} requires that~$\rho>\lambda_{\rm max}(Q)$. But in the several subsequent results, the assumption can be omitted. Hence, we prefer to formulate the dynamical scheme without the condition~$\rho>\lambda_{\rm max}(Q)$.}.

The next theorem addresses the existence and uniqueness of the trajectory of system~\eqref{dynamic_sys_A}.

\begin{theorem}\label{global_sol_A} {\rm (See~\cite[Theorem~2.1]{pty_2025})} For any $x^0\in\mathbb R^n$, $\rho>0$, and $\eta>0$, there exists a unique $C^1$ function $x:\mathbb{R}_+\to\mathbb R^n$ satisfying the differential equation and the initial condition in~\eqref{dynamic_sys_A}.
\end{theorem}

The set $C$ is \textit{flow invariant} with respect to  the trajectories $x(\cdot)$ generated by~\eqref{dynamic_sys_A}; that is, if $x^0\in C$ then $x(t)\in C$ for all $t\geq 0$. The precise formulation of the result is as follows.

\begin{theorem}\label{flow_invariant_A} {\rm (See~\cite[Theorem~2.3]{pty_2025})} For any $x^0\in C$, $\rho>0$, and $\eta>0$, the whole trajectory $x(\cdot)$ of~\eqref{dynamic_sys_A} is contained in $C$, that is $x(t)\in C$ for all $t\geq 0$. 
\end{theorem}

\section{Main Results}\label{Sect-3}

Our first main result shows that, under a mild assumption, the trajectory generated by~\eqref{dynamic_sys_A} converges asymptotically to the KKT set of~\eqref{QP problem}.

\begin{theorem}\label{thm1}
Let $\eta$ and $\rho$ be some positive constants. If the objective function $f(x)$ of problem~\eqref{QP problem} is bounded below on~$C$ then, for any $x^0\in C$, the unique solution $x(t)$ of~\eqref{dynamic_sys_A} has the property 
\begin{equation}\label{convergence_1}\lim\limits_{t\to +\infty} d(x(t),C_*) = 0.
\end{equation} Moreover, there exists a connected component $\Omega$ of $C_*$, which is a union of finitely many polyhedral convex sets (hence it is closed) such that 
\begin{equation}\label{convergence_1a}\lim\limits_{t\to +\infty} d(x(t),\Omega) = 0.
\end{equation} In particular, if problem~\eqref{QP problem} has finitely many KKT points, then there is a KKT point $\bar x$ of~\eqref{QP problem} such that $\lim\limits_{t\to +\infty} x(t)=\bar x$. 
\end{theorem}
\begin{proof} Since $x(\cdot)$ is the solution of~\eqref{dynamic_sys_A} starting from  $x^0\in C$, we have
	\begin{equation}\label{eq:projection}
	\eta	\dot{x}(t) +x(t)=  P_C\left(x(t)-\dfrac{1}{\rho}\big(Q x(t)+q\big)\right),
	\end{equation}
	for $t\geq 0$. By~\cite[Theorem~2.3, p.~9]{ks80}, the condition $z= P_C(u)$ is equivalent to the property 
	\begin{equation*}
		\langle u-z, y- z\rangle  \leq 0\quad \mbox{\rm for all}\; y\in C.
	\end{equation*}	
Thus, for every $t\geq 0$, by setting $u= x(t)-\dfrac{1}{\rho}\big(Q x(t)+q\big)$ and $z= \eta	\dot{x}(t) +x(t)$ we get from~\eqref{eq:projection} the inequality
	\begin{equation*}\label{dynamic_sys_A1}
		\left\langle\left(x(t)-\dfrac{1}{\rho}\big(Q x(t)+ q\big)\right)- \Big(\eta \dot{x}(t) + x(t)\Big),\; y-\Big(\eta \dot{x}(t) + x(t)\Big) \right\rangle \leq 0
	\end{equation*}
	for all $y\in C$. This gives	
	\begin{equation}\label{dynamic_sys_A2}
		\left\langle \eta \dot{x}(t) + \dfrac{1}{\rho}\big(Q x(t)+ q\big),\; y-\eta\dot{x}(t) -x (t)  \right\rangle \geq 0
	\end{equation}for all $y\in C$ and $t\geq 0$.

Thanks to Theorem~\ref{flow_invariant_A}, we know that $x(t)\in C$ for all $t\geq 0$. So, by substituting $y=x(t)$ to~\eqref{dynamic_sys_A2} and simplifying the expression, we get
	\begin{equation*}
		\eta \Vert\dot{x}(t)\Vert^2 + \dfrac{1}{\rho}\big\langle Q x(t)+ q, \dot{x}(t)  \big\rangle \leq 0 
	\end{equation*}
for all $t\geq 0$. Hence, for every $t\geq 0$, it holds that
	\begin{equation}\label{dynamic_sys_A32}
	\rho\eta\Vert\dot{x}(t)\Vert^2 + \big\langle Q x(t)+ q, \dot{x}(t)  \big\rangle \leq 0.
\end{equation}
 Since $f \big(x(t)\big)=  \left\langle \dfrac{1}{2}Qx(t)+q,x(t) \right\rangle$, taking the derivative with respect to $t$ and recalling that the matrix~$Q$ is symmetric, we have
\begin{equation}\label{f_new1}\begin{array}{rcl}
\dfrac{d}{dt}f \big(x(t)\big) &=& \left\langle \dfrac{1}{2}Q\dot{x}(t),x(t)\right\rangle + \left\langle \dfrac{1}{2}Qx(t)+q,\dot{x}(t) \right\rangle\\
&=&\left\langle \dot{x}(t),\dfrac{1}{2} Q^\top x(t)\right\rangle + \left\langle \dfrac{1}{2}Qx(t)+q,\dot{x}(t) \right\rangle\\
&=& \left\langle \dfrac{1}{2}(Q+Q^\top) x(t)+q,\dot{x}(t) \right\rangle\\
&=&\big\langle Qx(t)+q,\dot{x}(t) \big\rangle.\end{array}
\end{equation} By Theorem~\ref{global_sol_A}, the function $t\mapsto\dot{x}(t)$ is continuous on $\R_+$. Hence, the function \begin{equation}\label{f_new2}t\mapsto\big\langle Qx(t)+q,\dot{x}(t) \big\rangle\end{equation} is continuous on $\R_+$. Thanks to~\eqref{f_new1}, taking the Riemann integral of the function~\eqref{f_new2} and applying the Newton-Leibnitz formula, we get
$$\int\limits_{0}^{t}[\big\langle Qx(s)+q,\dot{x}(s) \big\rangle]ds= f\big(x(t)\big)-f\big(x(0)\big).$$
Therefore, integrating both sides of~\eqref{dynamic_sys_A32} over $[0, t]$ we obtain
$$\rho\eta \int\limits_{0}^{t}\Vert\dot{x}(s)\Vert^2ds + f\big(x(t)\big)-f\big(x(0)\big)\leq 0.$$
So, we have
	\begin{equation}\label{dynamic_sys_A52}
 \int\limits_{0}^{t}\Vert\dot{x}(s)\Vert^2 ds\leq \dfrac{1}{\rho\eta}\left[f(x_0)- 	f\big(x(t)\big)\right],
	\end{equation}
	for all $t\geq 0$. 
	
	To proceed furthermore, put $v^*=\inf\{f(x)\mid x\in C\}$ and note that $v^*\in\R$ by the assumption of the theorem. According to Theorem~\ref{flow_invariant_A}, one has $x(t)\in C$ for all $t\geq0$.  Hence, $	f\big(x(t)\big)\geq v^*$ for all $t \geq 0$.  It then follows from~\eqref{dynamic_sys_A52} that	
	\begin{equation}\label{dynamic_sys_A62}
 \int\limits_{0}^{t}\Vert\dot{x}(s)\Vert^2ds \leq\dfrac{1}{\rho\eta}\left[f(x_0)- 	v^*\right],\quad\mbox{\rm for all}\; t\in\R_+.
\end{equation}
Set $I(t) =  \displaystyle\int\limits_{0}^{t}\Vert\dot{x}(s)\Vert^2ds$ for $t\geq 0$. We see that the nonnegative function $I:\R_+\to\R$ is increasing and bounded above by the right-hand side of~\eqref{dynamic_sys_A62}. Thus, the value $I(t)$ converges to some number $M\in\R_+$ as $t\to+\infty$. In other words,
\begin{equation}\label{int_converges}
\int\limits_{0}^{+\infty}\Vert\dot{x}(s)\Vert^2ds=M.
\end{equation}

 Next, let us show that the function $\dot{x}(\cdot):\R_+\to\R^n$ is uniformly continuous. Put \begin{equation}\label{F(x)} F(x)=\dfrac{1}{\eta} \left[P_C\left(x-\dfrac{1}{\rho}\big(Q x+q\big)\right)-x\right]\quad \text{for } x\in \mathbb R^n.\end{equation}  As it has been shown in the proof of Theorem~2.1 in~\cite{pty_2025}, the function $F:\mathbb R^n\to \mathbb R^n$ is Lipschitz on $\R^n$ with the constant $$L:=\dfrac{1}{\eta}\left(\dfrac{1}{\rho}\|Q\|+2\right).$$From~\eqref{dynamic_sys_A} and~\eqref{F(x)} it follows that $\dot{x}(t)=F(x(t))$ for all $t\in\R_+$. Thus, for any $t,s\in \R_+$, we have
\begin{equation}\label{eq:dot_x_x}
	\|\dot{x}(t) - \dot{x}(s)\| = \|F(x(t)) - F(x(s))\| \leq L \|x(t) - x(s)\|.
\end{equation}
In addition, using the Newton-Leibnitz formula  and~\cite[Theorem~4(ii), p.~46]{DU_1977} gives
\begin{equation}\label{eq:xt_xs}	\|x(t) - x(s)\| = \left\| \int_s^t  \dot{x}(\tau) d\tau \right\| \leq \int_s^t  \|\dot{x}(\tau)\| d\tau
\end{equation}
for all $t,s\in\R_+$ with $s\leq t$.
Moreover, invoking the Cauchy-Schwarz inequality (see, e.g.,~\cite[Theorem~1.1, p.~100]{l1987}) for the normed space $L^2([s,t])$ equipped with its usual norm, from~\eqref{int_converges} we deduce that
\begin{equation} \label{eq:cauchy_schwarz}
	\begin{array}{ll}
\displaystyle	\int_s^t   \|\dot{x}(\tau)\| d\tau &\leq \left( \displaystyle\int_s^t  1^2 d\tau \right)^{1/2} \left( \displaystyle\int_s^t \|\dot{x}(\tau)\|^2 d\tau \right)^{1/2} \\
	&\leq  |t-s|^{1/2} \left( \displaystyle\int_0^{+\infty} \|\dot{x}(\tau)\|^2 d\tau \right)^{1/2} \\
	&=\sqrt{M}|t-s|^{1/2}.
	\end{array}
\end{equation}
Combining~\eqref{eq:dot_x_x},~\eqref{eq:xt_xs} and~\eqref{eq:cauchy_schwarz} yields
\begin{equation}\label{eq:holder_continuity}
	\|\dot{x}(t) - \dot{x}(s)\| \leq L  \sqrt{M}  |t-s|^{1/2}\quad {\rm for\ any}\ t,s\in\R_+.
\end{equation}
This means that $\dot{x}(\cdot)$ is H\"{o}lder continuous with exponent $\frac{1}{2}$ on $\R_+$. Given any $\varepsilon>0$, we choose $\delta>0$ as small as $L\sqrt{M}\delta^{1/2}<
\varepsilon$. Then, by~\eqref{eq:holder_continuity} we have $\|\dot{x}(t) - \dot{x}(s)\|<\varepsilon$ for any $t,s\in\R_+$ satisfying $|t-s|<\delta$. We have thus proved that the function $\dot{x}(\cdot)$ is uniformly continuous on $\R_+$.

\smallskip
\noindent{\sc Claim.} \textit{We have $\lim\limits_{t\to +\infty}\dot{x}(t)=0.$}

\smallskip
If the claim was false, then we would find a number~$\varepsilon>0$ and a sequence~$\{t_k\}$ converging to $+\infty$ such that $\|\dot{x}(t_k)\|\geq \varepsilon$. By the uniform continuity of  $\dot{x}(\cdot)$, there exists a number $\delta > 0$  such that for any $k \geq 1$ and any $t \in \left[t_k - \delta, t_k + \delta\right]$, we have $\|\dot{x}(t_k)-\dot{x}(t)\| < \dfrac{\varepsilon}{2}$. Hence, for every $k \geq 1$ and every $t \in \left[t_k - \delta, t_k + \delta\right]$, one has
\begin{align*}
\|\dot{x}(t)\|\geq \|\dot{x}(t_k)\|- \|\dot{x}(t_k)-\dot{x}(t)\|  >  \varepsilon -\dfrac{\varepsilon}{2}
= \dfrac{\varepsilon}{2}.
\end{align*}
 As $t_k\to+\infty$, we then can extract a  subsequence, which is still denoted by $\{t_k\}$, such that $t_1>\delta$ and $t_{k+1} - t_k > 2\delta$ for all $k \in \mathbb{N}$.  This implies that the intervals $I_k := \left[t_k - \delta, t_k + \delta\right]$, $k \in \mathbb{N}$, are pairwise disjoint. Therefore,
\begin{align*}
	\int_{0}^{+\infty} \|\dot{x}(s)\|^2 ds \ge \sum_{k=1}^{\infty} \int_{t_k - \delta}^{t_k + \delta} \|\dot{x}(s)\|^2 ds&\geq \sum_{k=1}^{\infty} \int_{t_k - \delta}^{t_k + \delta} \left(\dfrac{\varepsilon}{2}\right)^2 ds\\
	&=\sum_{k=1}^{\infty} a_k,
\end{align*}
where $a_k:=\dfrac{\delta \epsilon^2}{2}$ for all $k\in\N$. We then get
$\disp\int_{0}^{+\infty}\|\dot{x}(s)\|^2 ds = +\infty,$ which contradicts~\eqref{int_converges}. Thus, we must have
\begin{equation}\label{eq:dotx_to_0}
	\lim_{t\to+\infty}\dot{x}(t)=0.
\end{equation}

	 By our assumptions,  $f(x)$ is bounded below on~$C$ and $C$ is nonempty. So, according to the Frank-Wolfe theorem (see, e.g.,~\cite[Theorem~2.1]{lty05}), problem~\eqref{QP problem} has a solution. Hence, the set $C_*$ is nonempty (see Section~\ref{Sect-2}). Therefore, by Lemma~\ref{lem:dist_kkt_set}, there exist positive constants~$\varepsilon$ and~$\ell$ such that the estimate \eqref{ERB} holds for all $x\in C$ satisfying the condition~\eqref{locality}. 
	 
	  Thanks to~\eqref{eq:dotx_to_0}, we can find $\tau_0>0$ such that
	$$\|\dot{x}(t)\|< \dfrac{\varepsilon}{\eta}\quad\mbox{\rm for all}\; t\geq \tau_0.$$
Consequently, by the differential equation in~\eqref{dynamic_sys_A}, one has
	$$\left\|P_C\Big(x(t)-\frac{1}{\rho}(Qx(t)+q)\Big)-x(t)\right\|= \eta\|\dot{x}(t)\|< \varepsilon,$$
	for all $t\geq \tau_0$. This implies that, for every $t\geq \tau_0$, the element $x:=x(t)\in C$ satisfies the condition~\eqref{locality}. So, by~\eqref{ERB} we have 
	$$0\leq d(x(t),C_*)\leq \ell \left\|P_C\Big(x(t)-\frac{1}{\rho}(Qx(t)+q)\Big)-x(t)\right\| = \ell \eta \| \dot{x}(t)\|,$$
	for every $t\geq \tau_0$. As 
	$\lim\limits_{t\to+\infty}\dot{x}(t)=0$ by~\eqref{eq:dotx_to_0}, this yields~$\lim\limits_{t\to +\infty} d(x(t),C_*) = 0.$ We have thus proved that the property~\eqref{convergence_1} is valid. 
	
    According to Lemma~\ref{lem:KKT_components}, the KKT set $C_*$  of~\eqref{QP problem} has finitely many connected components, which are denoted by $C_1, C_2,\ldots, C_r$. In addition, we can find some $\delta >0$ such that~\eqref{positive_excess} holds for all $i, j$ belong to $I_r:=\{1,\ldots,r\}$ with $i\neq j$. For each $i\in I_r$, the distance function $d(\cdot, C_i)$ is globally Lipschitz (see~\cite[Proposition~2.4.1]{Clarke_1990}), and hence, it is continuous on $\R^n$. It follows that the sets
		$$V_i := \left\{x\in\R^n\mid d(x, C_i)<\dfrac{\delta}{3}\right\}\quad (i\in I_r)$$
	are open. Since~\eqref{positive_excess} holds for all $i, j\in I_r$ with $i\neq j$, the sets $V_i$, $i\in I_r$, are pairwise disjoint. 
	
	By~\eqref{convergence_1}, there is $\tau_1\in \R_+$ such that  $d(x(t),C_*)<\dfrac{\delta}{3}$ for every $t\geq \tau_1$. As $C_*=\bigcup\limits_{i\in I_r} C_i$, this implies that
	 $x(t)\in \bigcup\limits_{i\in I_r} V_i$ for every $t\geq \tau_1$. Thanks to the continuity of the function  $x:[\tau_1,+\infty)\to\R^n$, $t\mapsto x(t)$, and the fact that $V_i$, $i\in I_r$, are pairwise disjoint open sets, from the last property we can deduce the existence of an index $i_0\in I_r$ such that \begin{equation}\label{eq:x_in_Di0}
	 	x(t)\in V_{i_0}\qquad\text{for all }\, t\geq \tau_1.
	 \end{equation} Indeed, let $i_0$ be the unique index in $I_r$ such that $x(\tau_1)\in V_{i_0}$. Define the sets
	 $$U_1=\Big\{t\in [\tau_1,+\infty)\mid x(t)\in V_{i_0}\Big\},\quad U_2=\left\{t\in [\tau_1,+\infty)\mid x(t)\in \bigcup\limits_{i\in I_r\setminus\{i_0\}} V_i\right\}.$$ Note that $U_1$ and $U_2$ are disjoint subsets of the half line $[\tau_1,+\infty)$,  $U_1\cup U_2=[\tau_1,+\infty)$, and both sets are open in the induced topology of the half-line. Since $U_1\neq\emptyset$, we must have $U_2=\emptyset$, otherwise the half-line is disconnected, which is an absurd. Thus,~\eqref{eq:x_in_Di0} holds true.

	Now, fix any $t\geq \tau_1$. By~\eqref{eq:x_in_Di0} and the definition of $V_{i_0}$, one has $d(x(t),C_{i_0})<\dfrac{\delta}{3}$. Then, there exists $z\in C_{i_0}$ with $\|x(t)-z\|<\dfrac{\delta}{3}$. For any $i\in I_r\setminus\{i_0\}$ and any $y\in C_i$, by~\eqref{positive_excess} we have $\|z-y\|\geq d(z,C_i)\geq \delta$. Hence,
	\begin{equation*}
		\|x(t)-y\|\geq \|z-y\|-\|z-x(t)\|>\delta-\frac{\delta}{3}=\frac{2\delta}{3}.
	\end{equation*}
So, $$d(x(t),C_i)\geq \dfrac{2\delta}{3}> d(x(t),C_{i_0})$$ for every $i\in I_r\setminus\{i_0\}$. Since $d(x(t),C_*)=\min\{d(x(t),C_i)\mid i\in I_r\},$
	this implies that
	\begin{equation}\label{eq:dist_equal}
		d(x(t),C_*)=d(x(t),C_{i_0})\qquad\text{for all } t\geq \tau_1.
	\end{equation}
	Setting $\Omega=C_{i_0}$, from~\eqref{eq:dist_equal} and~\eqref{convergence_1}, we get
	\begin{equation*}
		\lim_{t\to+\infty} d(x(t),\Omega)=\lim_{t\to+\infty} d(x(t),C_{i_0})=\lim_{t\to+\infty} d(x(t),C_*)=0,
	\end{equation*}
	which establishes~\eqref{convergence_1a}.
	
	Finally, suppose that problem~\eqref{QP problem} has finitely many KKT points. Then, every connected component of $C_*$ must be a singleton. In particular, $\Omega=\{\bar x\}$
	for some $\bar x\in C_*$. Therefore, from~\eqref{convergence_1a} it follows that
	\begin{equation*}
		\lim_{t\to+\infty}\|x(t)-\bar x\|=\lim_{t\to+\infty} d(x(t),\Omega)=0,
	\end{equation*} which justifies the third assertion of the theorem. 
	\end{proof}

We have the following important corollary of Theorem~\ref{thm1}.

\begin{corollary}\label{cor1}
Under the assumptions of Theorem~\ref{thm1}, suppose in addition that $C$ is bounded. Then, there exists a subsequence of $\{x(t)\}_{t\geq 0}$ converging to a KKT point of~\eqref{QP problem}.
\end{corollary}

\begin{proof}
Since $C$ is closed and bounded, it is compact. Then, there exists a subsequence $\{x(t_k)\}_{k\in\N}$ of the sequence~$\{x(t)\}_{t\geq 0}$ which converges to some point $\bar x\in C$. By Theorem~\ref{thm1}, we have $\lim\limits_{t\to +\infty} d(x(t),C_*) = 0.$ This yields
$$\lim\limits_{k\to+\infty} d(x(t_k),C_*) = 0.$$
As the distance function to a nonempty set is globally Lipschitz~ (see~\cite[Proposition~2.4.1]{Clarke_1990}), it follows  that
 $$0=\lim\limits_{k\to+\infty} d(x(t_k),C_*)=d\left(\lim\limits_{k\to+\infty}x(t_k),C_*\right) = d(\bar x, C_*).$$
Since $C_*$ is closed by Lemma~\ref{lem:KKT_components},  from the latter we deduce that $\bar x\in C_*$.  Thus, subsequence $\{x(t_k)\}_{k\in\N}$ of~$\{x(t)\}_{t\geq 0}$ converges to a KKT point of~\eqref{QP problem}.
\end{proof}

\medskip
To investigate the convergence rate of the dynamical scheme~\eqref{dynamic_sys_A}, we need the next technical lemma.

\begin{lemma}\label{lem:quadratic_bound}
	Suppose that $Q\in\mathbb R^{n\times n}$  is a symmetric matrix and $\rho$ is a nonzero real number. Then, for any vectors  $x, y\in \R^n$, one has
	\begin{equation}\label{auxiliary_estimate_Q}
	\frac{1}{4\rho^2} \|Q(x-y)\|^2 - \frac{1}{\rho}\langle Q(x-y), x-y\rangle \leq \max_{\lambda \in \sigma(Q)} \left[ \frac{\lambda}{\rho} \left( \frac{\lambda}{4\rho} - 1 \right) \right] \|x - y\|^2
	\end{equation}
	with $\sigma(Q)$ being the set of the eigenvalues of $Q$.
\end{lemma}

\begin{proof}
Since $Q$ is symmetric, we have
	\begin{equation}\label{eq:M}
		\frac{1}{4\rho^2} \|Q(x-y)\|^2 - \frac{1}{\rho}\langle Q(x-y), x-y\rangle = (x-y)^\top \left( \frac{1}{4\rho^2} Q^2 - \frac{1}{\rho} Q \right) (x-y).
	\end{equation}
	Put $M = \dfrac{1}{4\rho^2} Q^2 - \dfrac{1}{\rho} Q$ and note that $M$ is a symmetric matrix. By~\cite[Theorem~1.1.6]{HJ_2002}, any eigenvalue $\mu$ of $M$ is given by 
	\begin{equation*}
		\mu = \frac{\lambda^2}{4\rho^2} - \frac{\lambda}{\rho}=\frac{\lambda}{\rho} \left( \frac{\lambda}{4\rho} - 1 \right),
	\end{equation*}
	where $\lambda$ is an eigenvalue of $Q$. So, combining this with an assertion of the Rayleigh theorem (see, e.g.,~\cite[Theorem~4.2.2(c)]{HJ_2002}), we have
	\begin{equation*}
 (x-y)^\top M (x-y)\leq \max_{\lambda \in \sigma(Q)} \left[ \frac{\lambda}{\rho} \left( \frac{\lambda}{4\rho} - 1 \right) \right] \|x - y\|^2.
	\end{equation*}
Hence, using~\eqref{eq:M}, we obtain the desired result.
\end{proof}

The distance between the trajectory of~\eqref{dynamic_sys_A} and a solution of~\eqref{QP problem} satisfies an estimate given by the next theorem.

\begin{theorem}\label{thm2}
Suppose that the assumptions of Theorem~\ref{thm1} are satisfied, $x^*$ is a solution of~\eqref{QP problem}, and $x(\cdot)$ is the unique solution of~\eqref{dynamic_sys_A} for some $x^0\in C$. Then, for every $t\in \R_+$ it holds that 
\begin{equation}\label{sol_bound}
\|x(t)-x^*\|\leq \exp\left(\dfrac{\alpha}{\eta} t\right)\|x^0-x^*\|,
\end{equation}
where \begin{equation}\label{alpha}\alpha := \max\limits_{\lambda \in \sigma(Q)} \left[ \frac{\lambda}{\rho} \left( \frac{\lambda}{4\rho} - 1 \right)\right].
\end{equation}
\end{theorem}

\begin{proof}
	 By our assumptions, $f$ is bounded below on $C$. Hence, using the Frank-Wolfe theorem (see, e.g.,~\cite[Theorem~2.1]{lty05}), we can assert that problem~\eqref{QP problem} has a solution~$x^*$. Substituting~$y=x^*$ into~\eqref{dynamic_sys_A2} yields
\begin{equation}\label{eq: est_kkt_1}
	\left\langle \eta \dot{x}(t) + \dfrac{1}{\rho}\big(Q x(t)+ q\big),\; x^*-\eta\dot{x}(t) -x (t)  \right\rangle \geq 0
\end{equation} for all $t\geq 0$. 
In addition, since $\eta\dot{x}(t) +x(t)=P_C\left(x(t)-\dfrac{1}{\rho}\big(Q x(t)+q\big)\right)\in C$ by~\eqref{eq:projection} and~$x^*\in C_*$, it follows from~\eqref{AVI} that
\begin{equation}\label{eq: est_kkt_2}
\dfrac{1}{\rho}\big\langle(Qx^*+ q), \eta\dot{x}(t) +x(t) - x^*\big\rangle \geq 0
\end{equation} 
 for every $t\in\R_+$. Adding inequalities~\eqref{eq: est_kkt_1} and~\eqref{eq: est_kkt_2} side by side gives
 $$\left\langle \eta \dot{x}(t) + \dfrac{1}{\rho}Q\big(x(t)-x^*\big),\; x^*-\eta\dot{x}(t) -x (t)  \right\rangle \geq 0$$
 for all $t\geq 0$. This is equivalent to
 $$\langle \eta \dot{x}(t),x(t)-x^*\rangle +\eta^2\|\dot{x}(t)\|^2+\dfrac{1}{\rho}\langle Q\big(x(t)-x^*\big),\eta \dot{x}(t)\rangle +\dfrac{1}{\rho}\left\langle Q\big( x(t)-x^*\big), x(t)-x^*\right\rangle\leq 0$$
 for all $t\in\R_+$.
 Then, by completing the square based on the third and forth terms of the sum on the left-hand side of the this inequality, we have
 $$\begin{array}{ll}
 &\langle \eta \dot{x}(t),x(t)-x^*\rangle + \left\|\eta \dot{x}(t) +\dfrac{1}{2\rho}Q\big(x(t)-x^*\big)\right\|^2-\dfrac{1}{4\rho^2}\|Q\big(x(t)-x^*\big)\|^2 \\
 &\quad+\dfrac{1}{\rho}\left\langle Q\big(x(t)-x^*\big), x(t)-x^*\right\rangle\leq 0
 \end{array}$$ 
 for all $t\geq 0$. Omitting the second term on the left-hand side of the last inequality and noting that $\langle \eta \dot{x}(t),x(t)-x^*\rangle= \dfrac{\eta}{2}\dfrac{d}{dt}\|x(t)-x^*\|^2$, we obtain
 \begin{equation*}
 	\dfrac{\eta}{2}\dfrac{d}{dt}\|x(t)-x^*\|^2 -\dfrac{1}{4\rho^2}\left\|Q\big(x(t)-x^*\big)\right\|^2  +\dfrac{1}{\rho}\left\langle Q\big(x(t)-x^*\big), x(t)-x^*\right\rangle\leq 0,
 \end{equation*}
 for all $t\geq 0$. Hence,
 
\begin{equation*}
	\dfrac{\eta}{2}\dfrac{d}{dt}\|x(t)-x^*\|^2 \leq\dfrac{1}{4\rho^2}\left\|Q\big(x(t)-x^*\big)\right\|^2  -\dfrac{1}{\rho}\left\langle Q\big(x(t)-x^*\big), x(t)-x^*\right\rangle
\end{equation*}
for all $t\geq 0$. Now, using the inequality~\eqref{auxiliary_estimate_Q} in 
 Lemma~\ref{lem:quadratic_bound} for $x=x(t)$ and $y=x^*$ gives
 \begin{equation*}
 	\dfrac{\eta}{2}\dfrac{d}{dt}\|x(t)-x^*\|^2 \leq\dfrac{1}{4\rho^2}\left\|Q\big(x(t)-x^*\big)\right\|^2  -\dfrac{1}{\rho}\left\langle Q\big(x(t)-x^*\big), x(t)-x^*\right\rangle\leq \alpha \|x(t) - x^*\|^2
 \end{equation*}
 for all $t \geq 0$, where $\alpha = \max\limits_{\lambda \in \sigma(Q)} \left[ \dfrac{\lambda}{\rho} \left( \dfrac{\lambda}{4\rho} - 1 \right)\right]$. Multiplying both sides of the resulted inequality $\dfrac{\eta}{2}\dfrac{d}{dt}\|x(t)-x^*\|^2 -\alpha \|x(t) - x^*\|^2\leq 0$ by $\dfrac{2}{\eta}\exp\left(-\dfrac{2\alpha}{\eta}t\right)$ yields
 $$\exp\left(-\frac{2\alpha}{\eta}t\right)\dfrac{d}{dt}\|x(t)-x^*\|^2-\frac{2\alpha }{\eta}\exp\left(-\frac{2\alpha}{\eta}t\right)\|x(t) - x^*\|^2\leq 0$$
 for all $t\geq 0$. This is equivalent to saying that
 $$\dfrac{d}{dt}\left[\exp\left(-\frac{2\alpha}{\eta}t\right)\|x(t)-x^*\|^2\right]\leq 0$$ for all $t\geq 0$. S, by the classical Newton-Leibnitz formula we have 
 $$0\geq\int_0^t  \left(\dfrac{d}{d\tau}\left[\exp\left(-\frac{2\alpha}{\eta}\tau\right)\|x(\tau)-x^*\|^2\right]\right)d\tau= \exp\left(-\frac{2\alpha}{\eta}t\right)\|x(t)-x^*\|^2-\|x(0)-x^*\|^2$$ for all $t\geq 0$. Hence, the inequality
$$\|x(t)-x^*\|^2\leq \exp\left(\frac{2\alpha}{\eta}t\right)\|x(0)-x^*\|^2$$ holds for every $t\in\R_+$.
Taking the square root of both sides of that inequality implies~\eqref{sol_bound}.
\end{proof}

The estimate~\eqref{sol_bound} for the convergence rate of the trajectory $x(t)$ depends greatly on the sign of the constant~$\alpha$ in~\eqref{alpha}. The latter can be determined explicitly by the correlation between the spectrum of $Q$ and the interval~$(0,4\rho)$ as follows.
\begin{lemma}\label{lem4}
	Let $\rho>0$ and $\alpha$ be given by~\eqref{alpha}. Then,
$$
	\begin{cases}
		\alpha <0, & \text{if } \sigma(Q)\subset(0,4\rho),\\[2mm]
		\alpha =0, & \text{if } \sigma(Q)\subset[0,4\rho] \text{ and } \sigma(Q)\cap\{0,4\rho\}\neq\emptyset,\\[2mm]
		\alpha >0, & \text{if } \sigma(Q)\not\subset[0,4\rho].
	\end{cases}
$$
\end{lemma}
\begin{proof}
	Put
	$\varphi(\lambda)=\frac{\lambda^2}{4\rho^2}-\frac{\lambda}{\rho}$ for $\lambda\in\mathbb{R}$, and note by~\eqref{alpha} that  $\alpha=\max\limits_{\lambda\in\sigma(Q)}\varphi(\lambda)$. Since $0$ and~$4\rho$ are the two roots of the quadratic polynomial $\varphi(\lambda)$ whose highest coefficient is positive, we have
	\begin{equation}\label{value_varphi}
		\begin{cases}
			\varphi(\lambda)<0 \ \text{ for } \lambda\in(0,4\rho), \\
			\varphi(0)=\varphi(4\rho)=0, \\
			\varphi(\lambda)>0 \ \text{ for } \lambda\in(-\infty,0)\cup(4\rho,\infty).
		\end{cases}
	\end{equation}
	
If $\sigma(Q)\subset(0,4\rho)$, then by~\eqref{value_varphi} one has $\alpha=\max\limits_{\lambda\in\sigma(Q)}\varphi(\lambda)<0$.

If $\sigma(Q)\subset[0,4\rho]$ and $\sigma(Q)\cap\{0,4\rho\}\neq\emptyset$ then, using~\eqref{value_varphi} again, one can deduce that~$\alpha=0$. 

If $\sigma(Q)\not\subset[0,4\rho]$, then there exists 
 $\bar\lambda\in\sigma(Q)$ with $\bar\lambda<0$ or there is  $\hat\lambda\in\sigma(Q)$ with $\hat\lambda>4\rho$. Since $\varphi(\bar\lambda)>0$ and $\varphi(\hat\lambda)>0$ by~\eqref{value_varphi}, one has
$
\alpha=\max\limits_{\lambda\in\sigma(Q)}\varphi(\lambda)>0.
$

The proof is complete.
\end{proof}

 Based on Theorem~\ref{thm2} and Lemma~\ref{lem4}, we have the next result.

\begin{corollary}\label{cor2_new}
	Under the assumptions of Theorem~\ref{thm2}, the following assertions are valid.
	 \begin{itemize}
		\item[{\rm (a)}] \text{If} $\sigma(Q)\subset(0,4\rho)$ \text{then}, for every $t\in \R_+$, one has the estimate~\eqref{sol_bound} with $x^*$ being the unique solution of~\eqref{QP problem} and $\alpha$ being defined by~\eqref{sol_bound}, and $\alpha<0$. In particular, $x(t)$ converges to $x^*$ when $t\to +\infty$.
		\item[{\rm (b)}] \text{If} $\sigma(Q)\subset[0,4\rho] \text{ and } \sigma(Q)\cap\{0,4\rho\}\neq\emptyset$ \text{then}, for every $t\in \R_+$, one has
		\begin{equation}\label{sol_bound2}
			\|x(t)-x^*\|\leq\|x^0-x^*\|.
		\end{equation}
			Thus, the whole trajectory $\{x(t)\}_{t\geq 0}$ lies in the ball $\bar B\big(x^*,\|x^0-x^*\|\big)$.
		\item[{\rm (c)}] \text{If} $\sigma(Q)\not\subset[0,4\rho]$, \text{then} $x(t)=x^*$ for all $t\in \R_+$ if $x^0=x^*$, and 
		\begin{equation}\label{sol_bound3}
			\dfrac{\|x(t)-x^*\|}{\|x^0-x^*\|}\leq \exp\left(\dfrac{\alpha}{\eta} t\right)
		\end{equation} for all $t\in \R_+$ if $x^0\neq x^*$, where  $\alpha$ is defined by~\eqref{sol_bound} and is positive.
	\end{itemize} 
\end{corollary}
\begin{proof} To proof assertion~(a), note that the condition $\sigma(Q)\subset(0,4\rho)$ implies that $Q$ is a positive definite matrix. Hence, problem~\eqref{QP problem} has a unique solution (see, e.g.,~\cite[Theorem~4.4]{lty05}). By Lemma~\ref{lem4}, one has $\alpha<0$. So, the desired properties follow from Theorem~\ref{thm2}.
	
	Assertion~(b) holds because if $\sigma(Q)\subset[0,4\rho]$ and $\sigma(Q)\cap\{0,4\rho\}\neq\emptyset$, then the constant $\alpha$ defined by~\eqref{sol_bound} is 0 (see Lemma~\ref{lem4}); so~\eqref{sol_bound} implies~\eqref{sol_bound2}.
	
	Finally, to prove assertion~(c), note that if $\sigma(Q)\not\subset[0,4\rho]$, then the constant $\alpha$ is positive by Lemma~\ref{lem4}. Hence, the two claims made in~(c) are immediate from Theorem~\ref{thm2}.	
\end{proof}

The following remarkable property holds: \textit{The value of the objective function of~\eqref{QP problem} is decreasing along the trajectory $\{x(t)\}_{t\geq 0}$ of~\eqref{dynamic_sys_A}.}

\begin{theorem}\label{thm:nonincreasing}
	Under the assumptions of Theorem~\ref{thm1}, one has
	\begin{equation}\label{eq:obj_funct_decrease}
		f\big(x(t_2)\big)\leq f\big(x(t_1)\big) -\rho\eta \int\limits_{t_1}^{t_2}\Vert\dot{x}(s)\Vert^2ds
	\end{equation}
	for every $t_1, t_2 \in\R_+$ such that $t_1\leq t_2$.
\end{theorem}
\begin{proof}
Take any $t_1, t_2\in\R_+$ with $t_1\leq t_2$. Arguing similarly as in the proof of Theorem~\ref{thm1}, we obtain~\eqref{dynamic_sys_A32} and~\eqref{f_new1}. Hence, thanks to~\eqref{f_new1} and the continuity of the function $t\mapsto \Vert\dot{x}(t)\Vert^2$, integrating both sides of~\eqref{dynamic_sys_A32} over $[t_1,t_2]$ we obtain
$$\rho\eta \int\limits_{t_1}^{t_2}\Vert\dot{x}(s)\Vert^2ds + f\big(x(t_2)\big)-f\big(x(t_1)\big)\leq 0.$$
This implies~\eqref{eq:obj_funct_decrease}.
We can thus deduce that the value of the objective function $f$ is decreasing along the trajectory~$x(\cdot)$.
\end{proof}

Now, let us present some remarks on the use of the above results for solving~\eqref{QP problem} numerically by the dynamical approach.

\begin{remark}\label{rem:1}
	\rm If the matrix $Q$ is positive definite, then the condition $\sigma(Q)\subset(0,4\rho)$ is satisfied when we choose $\rho\in \left(4^{-1}\lambda_{\rm max} (Q),+\infty\right)$. Therefore, by assertion~(a) of Corollary~\ref{cor2_new}, the trajectory $x(\cdot)$ converges exponentially to the unique solution of~\eqref{QP problem}. Given any $\varepsilon>0$ and suppose that $x^0\neq x^*$, from~\eqref{sol_bound} any the fact that $\alpha >0$ we can easily deduce that $\|x(t)-x^*\|\leq \varepsilon$ for all $t\geq \tau_\varepsilon$ with $ \tau_\varepsilon:=\dfrac{\eta}{\alpha}\ln\dfrac{\varepsilon}{\|x^0-x^*\|}$. In particular, solving the initial value problem~\eqref{dynamic_sys_A} just for $t\in[0,\tau_\varepsilon]$ one gets a piece of the trajectory of~\eqref{dynamic_sys_A} which satisfies the condition $\|x(\tau_\varepsilon)-x^*\|\leq \varepsilon$. If $x^0= x^*$ then~\eqref{sol_bound} forces $x(t)=x^*$ for all $t\geq 0$.
\end{remark}

\begin{remark}\label{rem:2}\rm
If $Q$ is positive definite, then by the first assertion of~\cite[Theorem~4.4]{lty05} we have $C_*=\{x^*\}$, where $x^*$ is the unique solution of~\eqref{QP problem}. So, for any $\rho >0$, the property~\eqref{convergence_1} in Theorem~\ref{thm1} implies that $\lim\limits_{t\to+\infty} x(t)=x^*$. This property is weaker than the exponential convergence of $x(\cdot)$ provided by  Remark~\ref{rem:1}, if $\rho$ is such that $\rho>\dfrac{1}{4}\lambda_{\rm max} (Q)$.
\end{remark}

\begin{remark}\label{rem:3}\rm
If $Q$ is positive semidefinite and $\lambda_{\rm min}(Q)=0$ then, for any  $\rho\geq\dfrac{1}{4}\lambda_{\rm max} (Q)$, the trajectory $\{x(t)\}_{t\geq 0}$ lies in a closed ball. Hence, there exists a sequence of positive numbers $\{t_k\}$ converging to $+\infty$ such that $\lim\limits_{k\to+\infty} x(t)=\bar x$ for some $\bar x \in C$. Since the distance function $d(\cdot, C_*)$ is Lipschitzian, using\eqref{convergence_1} gives
$$0=\lim\limits_{t\to +\infty} d(x(t),C_*) = \lim\limits_{k\to +\infty} d(x(t_k),C_*)= d(\bar x,C_*).$$
This shows that $\bar x$ is a KKT point of~\eqref{QP problem}. As $Q$ is positive semidefinite, by the third claim of~\cite[Theorem~4.4]{lty05} we can infer that $\bar x$ is a solution of~\eqref{QP problem}. 
\end{remark}

\begin{remark}\label{rem:4}\rm
	If $Q$ is not positive semidefinite then, for any $\rho >0$ one has $\sigma(Q)\not\subset[0,4\rho]$. Thus, by Corollary~\ref{cor2_new}(c), we have the estimate~\eqref{sol_bound3} for all $t\in\R_+$, provided that $x^0\neq x^*$. Obviously, this property does not imply~\eqref{convergence_1} and~\eqref{convergence_1a}.
\end{remark}

Among other things, the above four remarks shows that, despite to some interesting relationships, Theorem~\ref{thm1} and Theorem~\ref{thm2} are independent results.

\section{Examples}\label{Sect-4}
\subsection{First example}
This example shows how the conclusions of Theorem~\ref{thm1} manifest. Here, we have a nonconvex quadratic program with an unbounded set of KKT points.
	\begin{example}\label{ex:1}
		{\rm			
			For $n=2$, $Q=\begin{pmatrix}2&0\\0&-2\end{pmatrix}$, $q=b=(0,0)^\top$, $A=\begin{pmatrix}1&-1\\1&1\end{pmatrix}$, the IQP in~\eqref{QP problem} becomes minimizing $f(x)=x_1^2-x_2^2$ over the set $C=\{x=(x_1,x_2)\mid x_1\geq|x_2|\}.$ Fix an initial value $x^0=(x_1^0,x_2^0)\in C$ and take arbitrarily positive numbers $\eta$ and $\rho$. Note that $f$ is bounded below on $C$ and $v^*:=\inf\limits_{x\in C}f(x)=0$.  Hence, the assumptions of Theorem~\ref{thm1} are satisfied. The KKT point set is $C_*=E_+\cup E_-,$
			where
			$$E_+:=\{(t,t)\mid t\geq0\},\quad E_-:=\{(t,-t)\mid t\geq0\}.$$
			Clearly, $C_*$ coincides with the solution set of the IQP under our investigation. Since $C_*$ is connected, it has a unique connected component $\Omega=C_*$ which is the union of two polyhedral convex sets.
			
			Concerning the trajectory $x(t)=(x_1(t),x_2(t))$ of the dynamical system~\eqref{dynamic_sys_A}, observe that $$x(t)-\dfrac{1}{\rho}\big(Qx(t)+q\big)= \begin{pmatrix}\left(1-\dfrac2\rho\right)x_1(t)\\[1.5ex]\left(1+\dfrac2\rho\right)x_2(t)\end{pmatrix};$$
			 hence the differential equation becomes
			\begin{equation}\label{eq:ode_start}
				\dot x(t)=\dfrac{1}{\eta}\big[P_C\big(z(x(t))\big)-x(t)\big]
			\end{equation}
			with $z(x):=(\gamma x_1,\beta x_2)$, where $\gamma:=1-\dfrac2\rho$ and $\beta:=1+\dfrac2\rho$. We have $\gamma\in(-\infty,1)$, $\beta\in(1,+\infty)$, 
			\begin{equation}\label{eq:ab_identities_new}
			\beta = 1 + \frac{2}{\rho} > \left|1 - \frac{2}{\rho}\right| = |\gamma|
			\end{equation}
			and 
			\begin{equation}\label{eq:ab_identities}
				\gamma+\beta=2,\qquad \beta-\gamma=\frac4\rho>0.
			\end{equation}
			
			For any $u=(u_1,u_2)\in\mathbb R^2$, the projection of $(u_1,u_2)$ onto $C$ is given by
			\begin{equation}\label{eq:proj_form}
				P_C(u_1,u_2)=
				\begin{cases}
					(u_1,u_2), & u_1\geq |u_2|,\\[1mm]
					(0,0), & u_1\leq -|u_2|,\\[1mm]
					\dfrac{u_1+|u_2|}{2}(1,\operatorname{sgn} u_2), & |u_1|<|u_2|
					.
				\end{cases}
			\end{equation}
			 By Theorem~\ref{flow_invariant_A}, if $x(t)=(x_1(t),x_2(t))$ is the solution of~\eqref{eq:ode_start} with $x(0)=x^0$, then $x(t)\in C$ for all $t\geq 0$, i.e., $x_1(t)\ge|x_2(t)|$ for all $t\in\R_+$. 
			 
			 Let~$u^0=(u_1^0, u_2^0)=(\gamma x_1^0, \beta x_2^0)$. We consider the following three cases.
			
			\medskip
			\noindent\textbf{Case 1: $u_1^0\leq -|u_2^0|$}. Take
			\begin{equation}\label{eq:sol_formula_1}
				x(t)=x^0e^{-t/\eta},\quad \text{for all} \; t\geq 0.
			\end{equation}
			Thanks to the solution uniqueness stated in Theorem~\ref{global_sol_A}, we can use~\eqref{eq:proj_form} to verify that $x(\cdot)$ given by~\eqref{eq:sol_formula_1} is the solution of~\eqref{eq:ode_start} with $x(0)=x^0$. From~\eqref{eq:sol_formula_1} it follows that
			$$
			\lim_{t\to+\infty}x(t)=(0,0)\in C_*.
			$$
				
			\medskip
	
			\noindent\textbf{Case 2: $|u_1^0|<|u_2^0|$}. Recall that $u^0=(u_1^0, u_2^0)=(\gamma x_1^0, \beta x_2^0)$. Since $x^0\in C$, we see that if $x_1^0=0$, then $x_2^0=0$. This implies that $u_1^0= u_2^0 =0$, which is excluded in the case under consideration. So, we must have $x^0_1>0$.
			
			If $x_2^0 \geq 0$, then $|\gamma x_1^0| < |\beta x_2^0| = \beta x_2^0$. Let $\mu_1 = \dfrac{\gamma x_1^0+\beta x_2^0}{2}$ and 
			\begin{equation}\label{eq:ex_sol1}
				x_1(t)=\mu_1+\frac{\beta (x_1^0-x_2^0)}2e^{-t/\eta},\qquad x_2(t)= \mu_1-\frac{\gamma (x_1^0-x_2^0)}2e^{-t/\eta},
			\end{equation}
			for $t\in\R_+$. Set $\theta(t) = e^{-t/\eta}$. Then, $\theta(t) \in (0, 1]$ for all $t \in \mathbb{R}_+$. By~\eqref{eq:ab_identities} we get
			\begin{align*}
				x_1(t) = \mu_1 + \frac{\beta(x_1^0 - x_2^0)}{2}\theta(t) 
				&= (1 - \theta(t))\mu_1 + \dfrac{\gamma +\beta}{2}\theta(t) x_1^0\\
				&= (1 - \theta(t))\mu_1 + \theta(t) x_1^0
			\end{align*}
			and 	
			\begin{align*}
				x_2(t) = \mu_1 - \frac{\gamma(x_1^0 - x_2^0)}{2}\theta(t) 
				&= (1 - \theta(t))\mu_1 + \dfrac{\gamma +\beta}{2}\theta(t) x_2^0\\
				&= (1 - \theta(t))\mu_1 + \theta(t) x_2^0.
			\end{align*}
			It follows that
			$$
			\gamma x_1(t) = (1 - \theta(t)) \gamma \mu_1 + \theta(t) (\gamma x_1^0)\quad\text{and}\quad \beta x_2(t) = (1 - \theta(t)) \beta \mu_1 + \theta(t) (\beta x_2^0).
			$$
			Thus, 
			\begin{equation}\label{eq:beta_alpha}
				\begin{aligned}
					\beta x_2(t) - |\gamma x_1(t)|& = (1 - \theta(t)) \beta \mu_1 + \theta(t) (\beta x_2^0) - \big|(1 - \theta(t)) \gamma \mu_1 + \theta(t) (\gamma x_1^0) \big| \\&\geq (1 - \theta(t)) (\beta \mu_1 - |\gamma \mu_1|) + \theta(t) (\beta x_2^0 - |\gamma x_1^0|).
				\end{aligned}
			\end{equation}
			As $|\gamma x_1^0| < \beta x_2^0$, one gets $\gamma x_1^0 > - \beta x_2^0$. Hence,
			$$\mu_1 = \frac{\gamma x_1^0 + \beta x_2^0}{2} > \frac{-\beta x_2^0 + \beta x_2^0}{2} = 0.$$
			Therefore, from~\eqref{eq:ab_identities_new} it follows that
			$$\beta \mu_1 - |\gamma \mu_1| = \beta \mu_1 - |\gamma| \mu_1 = (\beta - |\gamma|)\mu_1 > 0.
			$$
			Since $\beta \mu_1 - |\gamma \mu_1| > 0$, $\beta x_2^0 - |\gamma x_1^0| > 0$, and $\theta(t) \in (0, 1]$ for all $t \in \mathbb{R}_+$, by~\eqref{eq:beta_alpha} we have
			$$\beta x_2(t) - |\gamma x_1(t)| > 0 
			$$
			for all $t \in \mathbb{R}_+$. So, $|\gamma x_1(t)| < \beta x_2(t)$ for all $t\in \R_+$. Since $\beta\in(1,+\infty)$, this implies that $x_2(t) >0$ for all $t\in\R_+$. Hence, by~\eqref{eq:proj_form} we get
			$$P_C\big(z(x(t))\big)= P_C(\gamma x_1(t), \beta x_2(t))=\dfrac{\gamma x_1(t)+\beta x_2(t)}{2}(1,1).$$ 
			This allows us to rewrite~\eqref{eq:ode_start} as
			$$\dot{x}(t)=\dfrac{1}{\eta} \left[\dfrac{\gamma x_1(t)+\beta x_2(t)}{2}(1,1)-x(t)\right].$$
			We can easily check that $x(\cdot)$ given by~\eqref{eq:ex_sol1} is the solution of~\eqref{eq:ode_start} with $x(0)=x^0$. So,~\eqref{eq:ex_sol1} implies that
			$$
			\lim_{t\to+\infty}x(t)=(\mu_1,\mu_1)\in E_+\subset C_*.	
			$$
			
		If $x_2^0 < 0$, then $|\gamma x_1^0| < |\beta x_2^0| = -\beta x_2^0$. Let $\mu_2 = \dfrac{\gamma x_1^0 - \beta x_2^0}{2}$ and define
		\begin{equation}\label{eq:ex_sol1_p2}
			x_1(t) = \mu_2 + \big(x_1^0 - \mu_2\big)e^{-t/\eta}, \qquad x_2(t) = -\mu_2 + \big(x_2^0 + \mu_2\big)e^{-t/\eta}.
		\end{equation}
	Put $\theta(t) = e^{-t/\eta}$ for $t \in \mathbb{R}_+$. We have
		$$
		x_1(t) = (1-\theta(t))\,\mu_2 + \theta(t)\,x_1^0, \qquad x_2(t) = -(1-\theta(t))\,\mu_2 + \theta(t)\,x_2^0.
		$$
		Therefore,
		$$
		\gamma x_1(t) = (1-\theta(t))\,\gamma \mu_2 + \theta(t)\,(\gamma x_1^0)
		$$
		and
		$$
		-\beta x_2(t) = (1-\theta(t))\,\beta \mu_2 + \theta(t)\,(-\beta x_2^0).
		$$
		Thus,
		\begin{equation}\label{eq:beta_alpha_p2}
			\begin{aligned}
				-\beta x_2(t) - |\gamma x_1(t)| &= (1-\theta(t))\,\beta \mu_2 + \theta(t)\,(-\beta x_2^0) - \big|(1-\theta(t))\,\gamma \mu_2 + \theta(t)\,(\gamma x_1^0)\big| \\
				&\geq (1-\theta(t))\big(\beta \mu_2 - |\gamma \mu_2|\big) + \theta(t)\big({-\beta x_2^0} - |\gamma x_1^0|\big).
			\end{aligned}
		\end{equation}
		Since $|\gamma x_1^0| < -\beta x_2^0$, one gets $\gamma x_1^0 > \beta x_2^0$. Hence,
		$$
		\mu_2 = \frac{\gamma x_1^0 - \beta x_2^0}{2} > \frac{\beta x_2^0 - \beta x_2^0}{2} = 0.
		$$
		Thus, it follows from~\eqref{eq:ab_identities_new} that
		$$
		\beta \mu_2 - |\gamma \mu_2| = \beta \mu_2 - |\gamma|\mu_2 = (\beta - |\gamma|)\mu_2 > 0.
		$$
		Since both $\beta \mu_2 - |\gamma \mu_2| > 0$ and $-\beta x_2^0 - |\gamma x_1^0| > 0$, and $\theta(t) \in (0,1]$ for all $t \in \mathbb{R}_+$, by \eqref{eq:beta_alpha_p2} we have
		$$
		-\beta x_2(t) - |\gamma x_1(t)| > 0
		$$
		for all $t \in \mathbb{R}_+$. So, $|\gamma x_1(t)| < -\beta x_2(t)$ for every $t \in \mathbb{R}_+$. As $\beta \in (1,+\infty)$, this implies $x_2(t) < 0$ for all $t \in \mathbb{R}_+$. Hence, by \eqref{eq:proj_form} we get
		$$
		P_C\big(z(x(t))\big) = P_C\big(\gamma x_1(t), \beta x_2(t)\big) = \frac{\gamma x_1(t) - \beta x_2(t)}{2}(1,-1).
		$$
		We can easily check that $x(\cdot)$ given by \eqref{eq:ex_sol1_p2} is the solution of \eqref{eq:ode_start}. So, we obtain
		$$
		\lim_{t\to+\infty} x(t) = (\mu_2, -\mu_2) \in E_- \subset C_*.
		$$
			
			\medskip
			\noindent\textbf{Case 3: $u_1^0\geq |u_2^0|$}. Since $u^0=(u_1^0, u_2^0)=(\gamma x_1^0, \beta x_2^0)$, if $x_1^0=0$ then we have $x_2^0=0$; hence $u_1^0= u_2^0 =0$. This situation has been considered in Case~1. So,  we may assume that $x^0_1>0$.
			
			If $0<\rho<2$, then $\gamma<0$. So, $u_1^0=\gamma x_1^0\leq 0\leq |u_2^0|$, with equality only if $x_2^0=0$ and $\gamma x_1^0=0$. Since $x_1^0>0$, this forces $\gamma=0$, which contradicts the fact that $\gamma<0$. Hence, this situation occurs only when $\rho\ge2$, i.e., $\gamma\geq 0$. 	Set $\sigma_0=\dfrac{x^0_2}{x^0_1}\in\left[-\dfrac{\gamma}{\beta},\dfrac{\gamma}{\beta}\right]$.
			
			\smallskip
			If $\sigma_0=0$, i.e., $x_2^0=0$, we can check that
			$$
			x(t):=\big(x_1^0e^{(\gamma-1)t/\eta},\,0\big),\qquad t\ge0,
			$$
			solves~\eqref{eq:ode_start} and $x(0)=x^0$. Since $\gamma<1$,
			$$
			\lim_{t\to+\infty}x(t)=(0,0)\in  C_*.
			$$
			
			\smallskip
			If  $0<\sigma_0\leq\dfrac{\gamma}{\beta}$, then define
			$$
			T_1=\frac\eta{\beta-\gamma}\ln\dfrac{\gamma}{\sigma_0\beta}
			$$
			and
			$$
			c_1:=\gamma x_1^0e^{(\gamma-1)T_1/\eta},\qquad c_2:=x_1^0e^{(\gamma-1)T_1/\eta}-x_2^0e^{(\beta-1)T_1/\eta}.
			$$
			Let
			\begin{equation}\label{eq:ex_sol2}
				x(t) = 
				\begin{cases}
					\Big(x_1^0e^{(\gamma-1)t/\eta},\ x_2^0e^{(\beta-1)t/\eta}\Big) \quad &\text{if}\;0\leq t\leq T_1,\\[2mm]
					\Big(c_1+\dfrac{\beta c_2}2e^{-(t-T_1)/\eta},\ \ c_1-\dfrac{\gamma c_2}2e^{-(t-T_1)/\eta}\Big) & \text{if}\;t > T_1.
				\end{cases}
			\end{equation}
			We see that $x(\cdot)$ is the unique solution of~\eqref{eq:ode_start} with $x(0)=x^0$. Indeed, we observe that $x(\cdot)$ is continuously differentiable on $\R_+$. For every $t\in[0,T_1]$, we have $$\gamma x_1(t)\geq\beta x_2(t)=|\beta x_2(t)|>0.$$ So, by~\eqref{eq:proj_form} one gets
			$$
			P_C\big(z(x(t))\big)=z(x(t))=(\gamma x_1(t),\beta x_2(t)),\qquad t\in[0,T_1].
			$$
			Since $\dot x_1(t)=\tfrac1\eta(\gamma-1)x_1(t)$ and $\dot x_2(t)=\tfrac1\eta(\beta-1)x_2(t)$ by direct differentiation, this is exactly $\dot x(t)=\tfrac1\eta[P_C(z(x(t)))-x(t)]$ on $[0,T_1]$. 
			
			For every $t > T_1$, we can infer that $|\gamma x_1(t)|<\beta x_2(t)$. By~\eqref{eq:proj_form}, for any $t> T_1$ this gives
			$$
			P_C\big(z(x(t))\big)=\dfrac{\gamma x_1(t)+\beta x_2(t)}{2}(1,1).
			$$
			It is easy to verify that $x(t)$ satisfies~\eqref{eq:ode_start} for $t>T_1$. Hence, the function $x(\cdot)$ given in~\eqref{eq:ex_sol2} is the solution of~\eqref{eq:ode_start} with the initial value $x^0=(x_1^0,x_2^0)$. Therefore,
			$$	\lim_{t\to+\infty}x(t) = (c_1,c_1) \in E_+\subset C_*.$$
			If $-\dfrac{\gamma}{\beta}\le \sigma_0<0$, then $x_2^0<0$. Define
			$$
			T_2:=\frac\eta{\beta-\gamma}\ln\dfrac{\gamma}{-\sigma_0\beta},
			$$
			and
			$$
			c_1:=\gamma x_1^0e^{(\gamma-1)T_2/\eta},\qquad c_2:=x_1^0e^{(\gamma-1)T_2/\eta}+x_2^0e^{(\beta-1)T_2/\eta}.
			$$
			Let
			\begin{equation*}\label{eq:ex_sol2_neg}
				x(t) = 
				\begin{cases}
					\Big(x_1^0e^{(\gamma-1)t/\eta},\ x_2^0e^{(\beta-1)t/\eta}\Big) \quad &\text{if}\;0\le t\le T_2,\\[2mm]
					\Big(c_1+\dfrac{\beta c_2}2e^{-(t-T_2)/\eta},\ \ -c_1+\dfrac{\gamma c_2}2e^{-(t-T_2)/\eta}\Big) & \text{if}\;t > T_2.
				\end{cases}
			\end{equation*}
			Arguing similarly as above, we obtain $x(\cdot)$ is continuously differentiable on $\mathbb{R}_+$. Therefore,
			$$
			\lim_{t\to+\infty}x(t) = (c_1,-c_1) \in E_-\subset C_*.
			$$
			
			We have seen that the  first two assertions of Theorem~\ref{thm1} are valid in each of the above three cases.
		}
	\end{example}

In Example~\ref{ex:1}, any trajectory of the dynamical system~\eqref{dynamic_sys_A} converges to a KKT point of~\eqref{QP problem}. This motivates the following open question:

\noindent\textit{Question 1: Does the unique trajectory $x(\cdot)$ of~\eqref{dynamic_sys_A} always have a cluster point?}

\subsection{Second example}
This example shows that the method of dynamical system, which is implemented in Algorithm~\ref{ContIQP} below, may outperform the well known DCA method.

\begin{algorithm}[H]
	\caption{ContIQP: Continuous method for solving IQPs}
	\label{ContIQP}
	\begin{algorithmic}[1]
		\Require Initial state $x_0 \in C$, time step $\Delta t > 0$, tolerance $\mathrm{Tol} > 0$, and constants $\eta > 0$ and $\rho > 0$
		\State Initialize $x \gets x_0$ and $t \gets 0$
		\Statex
		\While{$\|\dot{x}(t)\| \geq \mathrm{Tol}$}
		\State Solve the ODE~\eqref{dynamic_sys_A} over the interval $[t,\, t+\Delta t]$ with initial condition $x$ to obtain $x_{\mathrm{next}}$
		\State $x \gets x_{\mathrm{next}}$
		\State $t \gets t + \Delta t$
		\EndWhile
		\Statex
		\State \Return $x$ and $t$ 
	\end{algorithmic}
\end{algorithm}

\begin{example}\label{eg3.1} {\rm Consider problem~\eqref{QP problem} with $n=2, m = 4$, $$Q=\begin{bmatrix} 
			-3&2 \\ 2&-1
		\end{bmatrix},\ \; A=\begin{bmatrix}
			1&0 \\ 
			-1&0\\
			0&1\\
			0&-1
		\end{bmatrix},\ \; q=\begin{pmatrix}
			-2\\
			-2 
		\end{pmatrix},\ \; b=\begin{pmatrix}
			-2\\	
			-2\\
			-2\\
			-2
		\end{pmatrix}.$$	
	In other words, we consider the problem of minimizing the objective function $$f(x_1, x_2) = -\dfrac{3}{2}x_1^2 + 2x_1x_2 - \dfrac{1}{2}x_2^2 - 2x_1 - 2x_2$$ over the constraint set $$C=\big\{x\in\mathbb R^2 \mid -2\leq x_1\leq 2,\, -2\leq x_2\leq 2\big\}.$$ 
		Applying~\eqref{KKT_Point_Set}, we find that the set of KKT points is $$C_* = \left\{ \left(\dfrac{2}{3}, 2 \right), (2, 2), (2, -2), (-2, 2) \right\}.$$ Among these points, the global solutions are $(-2,2)$ and $(2,-2)$. 
		
		\begin{figure}[!ht]
			\centering
			\includegraphics[width=0.6\textwidth]{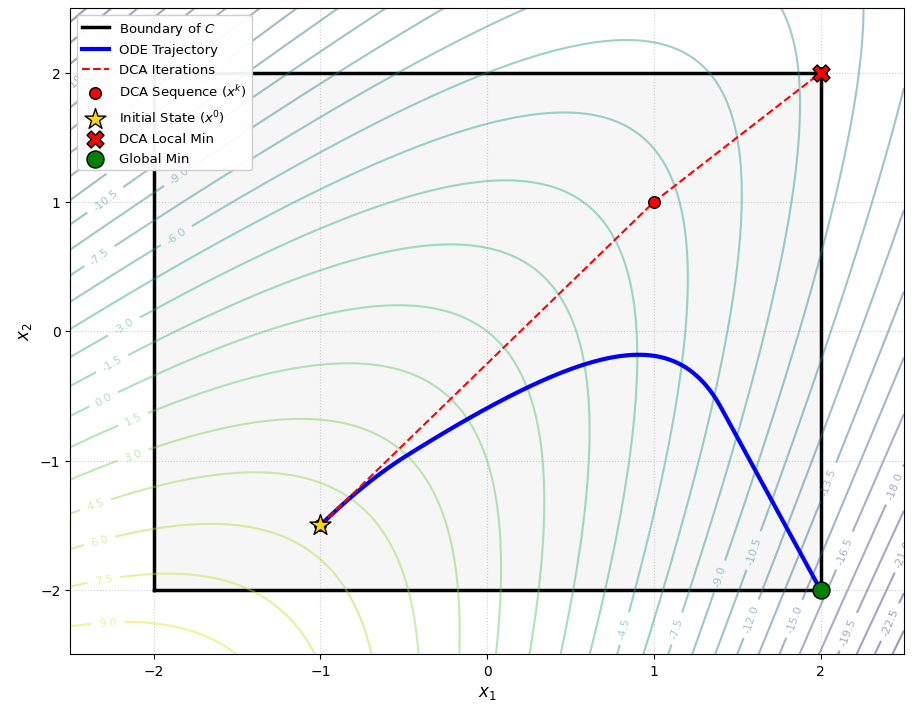} 
			\caption{Illustration for Example~\ref{eg3.1}}\label{fig1}
		\end{figure}
		
		We choose the initial point $
		x^0=\left(-1,-\dfrac{3}{2}\right),$
		and set the tolerance to ${\rm Tol}=10^{-6}$ with parameters $\eta=0.5$ and $\rho=1$. The corresponding numerical results are shown in Figure~\ref{fig1}. It can be seen that the classical DCA converges to the KKT point $(2,2)$, which is not a global solution.}
\end{example} 

\begin{remark}
	It is worthy to note that the discrete DCA scheme and its continuous-time
	counterpart need not converge to the same critical point when $Q$ is
	indefinite. The discrete iteration $x^{k+1} = F(x^k),$
where $F(x) := P_C\!\left(x - \tfrac{1}{\rho}\nabla f(x)\right)$, advances by  $F(x^k) - x^k$ at each step, its
	trajectory can traverse a substantial fraction of the feasible region $C$ in a
	single iteration and
	admitting convergence to whichever stationary point lies along the direction.
		
	On the other hand, the continuous scheme $
	\dot{x}(t) = \tfrac{1}{\eta}\big(F(x(t)) - x(t)\big)$
	behaves differently in a way that follows directly from the variational
	characterization of $F$. Because the vector field is integrated with a step size $\Delta t$ several orders of magnitude smaller than the diameter of $C$, the trajectory continuously descends through the constraint set rather than leaping across it.
\end{remark}

\section{Concluding Remarks}\label{Sect-5}
In this paper, we have investigated a continuous-time dynamical system for solving indefinite quadratic programming problems under linear constraints. We established the asymptotic convergence of the generated trajectory to the Karush-Kuhn-Tucker set of the problem. We also derived an estimate of the distance between the trajectory and a solution of the IQP. Furthermore, we showed that the objective function is decreasing along the trajectory. Illustrative examples were presented to demonstrate the convergence properties and numerical performance of the method.

Several directions deserve further investigation. An important topic is the development of efficient numerical discretizations of the continuous dynamical system and a comparison of their performance with existing DC algorithms for indefinite quadratic programming.

It would also be of interest to analyze, in the same manner, the second dynamical system proposed in~\cite{pty_2025}.

Finally, it is natural to ask whether it is possible to construct a continuous dynamical system whose discretization yields a discrete scheme with an inertial effect.

\end{document}